\documentclass[11pt]{amsart}

\usepackage{amsmath,amsthm} 
\usepackage{amsfonts} 
\usepackage{amssymb} 
\usepackage{mathrsfs}

\usepackage{eucal} 
\usepackage{bbm} 

\usepackage{enumerate}

\usepackage[overload, ntheorem]{empheq}

\makeatletter
\@namedef{subjclassname@2020}{%
  \textup{2020} Mathematics Subject Classification}
\makeatother

\usepackage{enumerate} 

\newtheorem{theorem}{Theorem}
\numberwithin{theorem}{section}
\newtheorem{lemma}[theorem]{Lemma}
\newtheorem{corollary}[theorem]{Corollary}
\newtheorem{proposition}[theorem]{Proposition}

\theoremstyle{definition}

\theoremstyle{definition}
\newtheorem{example}[theorem]{Example}
\newtheorem{remark}[theorem]{Remark}

\newtheorem*{problem*}{Problem}

\theoremstyle{plain}
\newtheorem*{theorema*}{Theorem A}
\newtheorem*{theoremb*}{Theorem B}

\def\enquote#1{``#1''}

\renewcommand{\mathcal}{\EuScript}

\newcommand{\N}{\mathbb{N}} 
\newcommand{\Z}{\mathbb{Z}}

\newcommand{\R}{\mathbb{R}}
\newcommand{\C}{\mathbb{C}}
\newcommand{\T}{\mathbb{T}}

\newcommand{\LLL}{\mathscr{L}}
\newcommand{\sgS}{\mathcal{S}}
\newcommand{\sgT}{\mathcal{T}}
\newcommand{\Id}{I}
\newcommand{\spec}{\sigma}
\newcommand{\fix}{\operatorname{fix}}
\newcommand{\lin}{\operatorname{lin}}
\newcommand{\rg}{\operatorname{rg}}
\newcommand{\defeq}{\mathrel{\vcentcolon=}}
\newcommand{\ce}{\mathrm{c}}
\newcommand{\co}{\ce_0}
\newcommand{\one}{\mathbbm{1}}
\newenvironment{abc}{\begin{enumerate}[{\rm(a)}]}{\end{enumerate}}
\newenvironment{iiv}{\begin{enumerate}[{\rm(i)}]}{\end{enumerate}}
\newenvironment{num}{\begin{enumerate}[{\rm1.}]}{\end{enumerate}}

\usepackage{marginnote}
\usepackage{xcolor}

\def\defi#1{\emph{#1}}

\allowdisplaybreaks[2] 
\begin{document}
	\title[Relative compactness of orbits and geometry of Banach spaces]{Relative compactness of orbits and geometry of Banach spaces}
	\author{B\'alint Farkas}
	\author{Henrik Kreidler}
	\address{University of Wuppertal, School of Mathematics and Natural Sciences, Gau\ss stra\ss e 20, 42119 Wuppertal, Germany}
	\email{farkas@uni-wuppertal.de}
	\email{kreidler@uni-wuppertal.de}
	
\begin{abstract}
We investigate for a bounded semigroup of linear operators $\sgS$ on a Banach space $E$ and a vector $x \in E$, when relative compactness of $\sgS(\Id-T)x$ for every $T \in \sgS$ implies relative compactness of the orbit $\sgS x$. In particular, we derive characterizations of separable Banach spaces not containing $\co$ and of reflexivity of Banach spaces with a Schauder basis in terms of such compactness results.
\end{abstract}
\maketitle

\section{Introduction}
A famous result due to V.~P.~Fonf, M.~Lin and P.~Wojtaszczyk (see \cite{FLW01}) states that a Banach space $E$ with a Schauder basis is reflexive if and only if every power-bounded linear operator $T \in \LLL(E)$ on $E$ is mean ergodic. Analogous results have been proved  by A.{} A.{} Albanese, J.{} Bonet and W.{}J.{} Ricker for operators  \cite{ABRop} and for one-parameter semigroups \cite{ABRsgrp} on  Fr\'echet spaces, a precursor of the latter characterization in the case of one-parameter semigroups on Banach spaces being due to D.{} Mugnolo \cite{Mugnolo}. We also refer to \cite{FLRUnif} for connections between uniform ergodicity and geometric properties of Banach spaces.

\medskip\noindent  One of the main goals of this article is to prove the following, different, operator theoretic characterization of reflexivity of Banach spaces with a Schauder basis.
\begin{theorema*}
	For a Banach space $E$ with a Schauder basis the following assertions are equivalent:
		\begin{abc}
			\item $E$ is reflexive.
			\item Every power-bounded operator $T\in\LLL(E)$ satisfying that the set of consecutive differences of the iterates
				\begin{align*}
					\{T^{n+1}x - T^{n}x\mid n \in \N\}
				\end{align*}
			is relatively compact for every $x \in E$, is \defi{almost periodic}, i.e., has relatively compact orbits 
			\begin{align*}
				\{T^{n}x\mid n \in \N\}\quad (x\in E).
			\end{align*}
		\end{abc}
\end{theorema*}
We also obtain a characterization of the larger class of separable Banach spaces $E$ not containing $\co$ if we consider \defi{doubly power-bounded} operators $T$  only (i.e., $T$ is invertible and both $T$ and $T^{-1}$ are power-bounded).
\begin{theoremb*}
	For a separable Banach space $E$ the following assertions are equivalent:
		\begin{abc}
			\item $E$ contains no isomorphic copy of $\co$.
			\item Every doubly power-bounded operator $T$ satisfying that the set of consecutive differences of the iterates
				\begin{align*}
					\{T^{n+1}x - T^{n}x\mid n \in \N\}
				\end{align*}
			is relatively compact for every $x \in E$, is almost periodic.			
		\end{abc}
\end{theoremb*}
The last characterization is motivated by a result of R.~B.~Bazit for functions $f \colon G \to E$ defined on a group $G$ and taking values in a Banach space $E$. He proved in \cite{Bolis71} that if $F \colon  G \to E$ is a bounded function such that
	\begin{num}
		\item all difference functions $F_t$ for $t \in G$ defined by
	\begin{align*}
		F_t(s)\defeq F(st) - F(s) \textrm{ for every } s \in G
	\end{align*}	 
			are almost periodic, and
		\item $E$ does not contain an isomorphic copy of $\co$,
	\end{num}
then $F$ is almost periodic. Here, a function $g\colon G\to E$ is called \defi{almost periodic} if $\{g(t \:\cdot\:)\mid t\in G\}$ is relatively compact in the Banach space of bounded, $E$-valued functions on $G$. This result---which is closely related to the classical Bohl--Bohr theorem and its generalization by M.~I.~Kadets to vector-valued functions in \cite{KadecMI}---hints at a connection between geometric properties of the Banach space and criteria for almost periodicity of operators. 

Given a doubly power-bounded operator $T \in \LLL(E)$  and $x \in E$ we can apply Bazit's result to the
function $f\colon  \Z \to E,\, k \mapsto T^kx$ and use Lemma \ref{lem:A1Am} from below to deduce the relative compactness of
	\begin{align*}
		\{T^{n}x\mid n \in \N\}
	\end{align*}
from the relative compactness of
	\begin{align*}
		\{T^{n+1}x - T^{n}x\mid n \in \N\}
	\end{align*}
if $E$ does not contain a copy of $\co$. 

However, we are primarily interested in the case of arbitrary power-bounded operators and, more generally, of bounded, abelian operator semigroups (in particular, bounded one-parameter operator semigroups).  More precisely, we discuss the following question.
\begin{problem*}\label{prob:main}
	Given a complex Banach space $E$, a bounded, abelian semigroup $\sgS \subseteq \LLL(E)$ of operators on $E$ and a vector $x \in E$, under what condition can we conclude relative compactness of the orbit
	\begin{align*}
		\sgS x = \{Sx\mid  S \in \sgS\}\subseteq E
	\end{align*}
from the relative compactness of the differences of the orbit points
	\begin{align*}
		\sgS(\Id -T)x &= \{Sx - STx\mid S \in \sgS\} 
	\end{align*}
for every $T \in \sgS$?
\end{problem*}

We focus on three aspects of this problem. First, we consider the case of mean ergodic semigroups in Section 2. We then look at semigroups on Banach spaces not containing $\co$ in Section 3. Finally, we  conclude with the stated characterization of reflexivity for Banach spaces admitting a Schauder basis in Section 4.

Before we start, we make the following simple but nevertheless useful observation, which makes things easier when dealing with bounded, finitely generated semigroups (in particular, cyclic semigroups generated by a single power-bounded operator).
\begin{lemma}\label{lem:A1Am}
Let $E$ be a Banach space and let $\sgS$ be the operator semigroup generated by the commuting, power-bounded operators $T_1,\dotsc,T_m\in\LLL(E)$. For $x \in E$ the following assertions are equivalent:
	\begin{abc}
		\item $\sgS(\Id-T_j)x$ is relatively compact for every $j \in \{1,\dotsc,m\}$.
		\item $\sgS(\Id-T)x$ is relatively compact for every $T \in \sgS$.
	\end{abc}
\end{lemma}
 \begin{proof}
Suppose that (a) holds and take an element $T \in \sgS$. We then find $k_1,\dotsc,k_m \in \N$ such that
 	\begin{align*}
 		T = \prod_{j=1}^m T_j^{k_j}.
 	\end{align*}
 But then
 	\begin{align*}
 		\Id - T &= \sum_{j=1}^m \prod_{i=1}^{j-1} T_i^{k_i} (\Id- T_j^{k_j})
 		= \sum_{j=1}^m \prod_{i=1}^{j-1} T_i^{k_i} \sum_{l=1}^{k_j} T_j^l (\Id- T_j)\\
 		&= \sum_{j=1}^m \sum_{l=1}^{k_j} \prod_{i=1}^{j-1} T_i^{k_i} T_j^l (\Id- T_j)
 	\end{align*}
 and therefore $\sgS(\Id-T)x$ is contained in a finite sum of relatively compact sets, and assertion (b) follows. The other implication is trivial.
 \end{proof}

\section{Mean ergodic semigroups}
It is considerably simpler to give answers to the proposed problem if we deal with mean ergodic semigroups. Recall that an abelian operator semigroup $\sgS \subseteq \LLL(E)$ is called \emph{mean ergodic} if one of the following equivalent conditions is satisfied:
\begin{abc}
	\item $\fix(\sgS)$ separates $\fix(\sgS')$.
	\item The closed convex hull of $\sgS$ with respect to the strong operator topology contains a projection $P \in \LLL(E)$ with
		\begin{align*}
			PS = SP = P
		\end{align*}
		for all $S \in \sgS$.
	\item The Banach space $E$ decomposes into the direct sum
		\begin{align*}
			\fix(\sgS) \oplus \overline{\rg(\Id-\sgS)}.
		\end{align*}
\end{abc}
Here,
	\begin{align*}
		\fix(\sgS) &\defeq \{x \in E\mid Sx = x \textrm{ for every } S \in \sgS\}, \\
		\rg(\Id - \sgS) &\defeq \lin \bigcup_{S \in \sgS} \rg(\Id - S),
	\end{align*}
and $\sgS'=\{S'\mid S\in\sgS\} \subseteq \LLL(E')$ denotes the semigroup of the adjoint operators.
Proofs for the above equivalences and further equivalent conditions involving the convergence of ergodic nets can be found in \cite[Corollary 1.8]{Schr2013} (see also \cite{Nage1973}).

The first partial solution to our problem reveals a connection between mean ergodicity and relative compactness of orbits of bounded, abelian operator semigroups.
\begin{theorem}\label{thm:charmeanergodic1}
	For a bounded, abelian operator semigroup $\sgS \subseteq \LLL(E)$ on a Banach space $E$ the following assertions are equivalent:
	\begin{abc}
		\item $\sgS$ has relatively compact orbits on $E$.
		\item The following two conditions are satisfied:
			\begin{iiv}
				\item $\sgS$ is mean ergodic.
				\item $\sgS$ has relatively compact orbits on $\rg(\Id - T)$ for every $T \in \sgS$.
			\end{iiv} 
	\end{abc}
\end{theorem}
\begin{proof}
	Since abelian semigroups with relatively (weakly) compact orbits are mean ergodic (see, e.g., Corollary 1.9 in \cite{Schr2013}), the implication \enquote{(a) $\Rightarrow$ (b)} is clear. On the other hand, if (b) holds, then $\sgS$ has relatively compact orbits on the space $\rg(\Id - \sgS)$ and then even on its closure. Since $\sgS$ has relatively compact orbits on $\fix(\sgS)$, it then has relatively compact orbits on
		\begin{align*}
			E = \fix(\sgS) \oplus \overline{\rg(\Id- \sgS)}.
		\end{align*}
	Thus the implication  \enquote{(b) $\Rightarrow$ (a)} is proved.
\end{proof}
\begin{corollary}\label{cor:charmeanergodic2}
	Let $\sgS \subseteq \LLL(E)$ be a bounded, mean ergodic, abelian operator semigroup on a Banach space $E$. For $x \in E$ the following assertions are equivalent:
		\begin{abc}
			\item $\sgS x$ is relatively compact.
			\item $\sgS(\Id-T)x$ is relatively compact for every $T \in \sgS$.
		\end{abc}
\end{corollary}
\begin{proof}
	Assume that (b) holds. Replacing $E$ by the closed linear hull of $\sgS x$ we may assume that $\sgS$ has relatively compact orbits on $\rg(\Id - T)$ for every $T \in \sgS$. But then the result directly follows from Theorem \ref{thm:charmeanergodic1}.
\end{proof}
Since bounded, abelian operator semigroups on reflexive Banach spaces are mean ergodic (by \cite[Corollary 1.9]{Schr2013}), we obtain the following consequence of Corollary \ref{cor:charmeanergodic2}.
\begin{corollary}\label{cor:bddsgrprefl}
	For a bounded, abelian semigroup $\sgS \subseteq \LLL(E)$ on a reflexive Banach space $E$ and $x \in E$ the following assertions are equivalent:
	\begin{abc}
		\item $\sgS x$ is relatively compact.
		\item $\sgS(\Id- T)x$ is relatively compact for every $T \in \sgS$. 
	\end{abc}
\end{corollary}

\begin{remark}\label{spectralcrit}
Let $T\in \LLL(E)$ be a power-bounded operator on a reflexive Banach space $E$. 
	\begin{abc}
		\item By Lemma \ref{lem:A1Am} and Corollary \ref{cor:bddsgrprefl} the set $\{T^n\mid n\in \N\}\subseteq E$ is relatively compact with respect to the strong operator topology if and only if $\{T^n(\Id-T)\mid n\in \N\}\subseteq E$ is relatively compact with respect to the strong operator topology.
		\item We obtain the following interesting spectral condition for relatively compact orbits: If $T$ has peripheral spectrum $\spec(T)\cap\T\subseteq\{1\}$, then $T$ has relatively compact orbits. Indeed, this follows from (a) and the Katznelson--Tzafriri theorem, the latter stating that under this spectral condition one has $\|T^n(\Id-T)\|\to 0$ as $n\to \infty$, see \cite{KaTz86}. In fact, by the mean-ergodic decomposition, in this case we even have that $(T^n)$ converges in the strong operator topology to the projection onto $\fix(T)$.
		\item Also note the following related observation brought to our attention by Jochen Glück: Suppose that $\spec(T)\cap\T$ is countable. Consider the Jacobs--de Leeuw--Glicks\-berg decomposition  $E=E_{\mathrm{aws}}+E_{\mathrm{rev}}$ into $T$-invariant closed, subspaces, where $T$ has relatively compact orbits on $E_{\mathrm{rev}}$ and $T|_{E_{\mathrm{aws}}}$ has no unimodular eigenvalues, see \cite[Thm.{} 9]{JonesLinGdL} for a short proof. By the Arendt--Batty theorem, see \cite[Thm.{} 5.1, Rem.{} 5.2]{AB}, we conclude that $(T^n)$ converges strongly to $0$ on $E_{\mathrm{aws}}$ altogether yielding that $T$ has relatively compact orbits on the entire space $E$. This spectral criterion for almost periodicity has also been noted independently by Lin, see \cite[Corollary 2.7]{Lin2020}. The analogous statement for one-parameter semigroups  is due to Lyubich and Vu \cite{LV}, see also Thm.{} 5.5.6 in \cite{ABHN}.\end{abc}
\end{remark}
\section{Semigroups on Banach spaces not containing $\co$}
As Bazit's result already suggests, a good condition for obtaining the desired equivalence for the relative compactness of orbits is that $E$  does not contain an isomorphic copy of $\co$. In fact, we have  the following example showing that on $\co$ there is an invertible isometry $T$ which is not almost periodic on $\co$, but on $\rg(\Id-T)$. We exhibit this example on the Banach space $\ce$ of convergent sequences, being isomorphic to $\co$.
\begin{example}\label{exa:1}
	Let $(a_n)_{n \in \N}$ be a sequence in $\C$ converging to $1$ such that $|a_n| =1 \neq a_n$ for all $n \in \N$. Consider the operator $T \in \LLL(\ce)$ on the space $\ce$ of convergent complex sequences defined by
		\begin{align*}
			T(x_n)_{n \in \N} \defeq (a_nx_n)_{n \in \N} \textrm{ for } (x_n)_{n \in \N} \in \ce.
		\end{align*}
		Then $T$ is an invertible isometry and $\fix(T) = \{0\}$. The linear functional
			\begin{align*}
				\ce\to\C, \quad (x_n)_{n \in \N} \mapsto \lim_{n \to\infty} x_n
			\end{align*}
		is contained in $\fix(T')$ and therefore $T$ is not mean ergodic. On the other hand, we obtain $\rg(\Id -T^n) \subseteq \co$ for every $n \in \N$ and $T$ restricted to $\co$ has relatively compact orbits.
\end{example} 
To enlighten what may be true for Banach spaces not containing $\co$, let us recall a result of W.{} Ruess and W.{} Summers, who considered the generalization of the Bohl--Bohr--Kadets result for vector-valued functions on $[0,\infty)$, see \cite[Thm.~2.2.2]{RuSu1} or \cite[Thm.~4.3]{RuSu2}.

Given an \defi{asymptotically almost periodic function} $f \colon  [0,\infty) \to E$ to a Banach space $E$, i.e., $\{f(\cdot+t )\mid t\geq 0\}$ is relatively compact in the Banach space of bounded, $E$-valued functions, one can find a unique almost periodic function $f_{\mathrm{r}}\colon  \R \to E$ and a unique continuous function $f_{\mathrm{s}}\colon  [0,\infty) \to E$ vanishing at infinity such that $f = f_{\mathrm{r}}|_{[0,\infty)} + f_{\mathrm{s}} $ (see \cite{Frechet}, \cite{FrechetProof}; and \cite{Fan} for functions defined on $\N$). Then the integral function
	\begin{align*}
		F \colon  [0,\infty) \to\C,\quad t \mapsto \int_0^t f(s)\, \mathrm{d}s
	\end{align*}
	is asymptotically almost periodic if $E$ does not contain $\co$ and $f_s$ is improperly Riemann integrable. As we see, a Jacobs--de Leeuw--Glicksberg type decomposition plays an essential role here. In our setting we use the following version (see \cite[Theorem 4.11]{dLGl1961} and \cite[Section 16.3]{EFHN}).
	\begin{theorem}\label{thm:jdlg}
		Let $E$ be a Banach space and $\sgS \subseteq \LLL(E)$ an abelian operator semigroup with relatively compact orbits and $\overline{\sgS}$ its closure with respect to the strong operator topology. Then there is a unique projection $P \in \overline{\sgS} \subseteq \LLL(E)$ such that 
			\begin{itemize}
				\item $\overline{\sgS}$ restricted to the range $E_{\mathrm{rev}}\defeq \rg(P)$ is a compact topological group, and
				\item $\overline{\sgS}$ restricted to the kernel $E_{\mathrm{aws}}\defeq \ker(P)$ contains the zero operator.
			\end{itemize}
	\end{theorem}
	In particular, if $\sgS$ is any bounded, abelian operator semigroup on a Banach space $E$, we obtain a Jacobs--de Leeuw--Glicksberg decomposition of the closed subspace
	\begin{align*}
		E_{\mathrm{aap}}\defeq \{x \in E \mid \sgS x \textrm{ is relatively compact}\}\subseteq E.
	\end{align*}
We denote the corresponding projection in $\LLL(E_{\mathrm{aap}})$ by $P$.
	\begin{theorem}\label{thm:mainthm}
		Let $\sgS \subseteq \LLL(E)$ be a bounded, abelian operator semigroup on a Banach space $E$ not containing an isomorphic copy of $\co$, and let $P\in \LLL(E_{\mathrm{aap}})$ be the corresponding  Jacobs--de Leeuw--Glicksberg projection. For $x \in E$ the following assertions are equivalent:
	\begin{abc}
		\item $\sgS x$ is relatively compact.
		\item The following conditions are satisfied:
		\begin{iiv}
			\item $\sgS(\Id -T)x$ is relatively compact for every $T \in \sgS$.
			\item $(\Id-P)(\Id - \sgS)x$ is relatively compact.
		\end{iiv}
	\end{abc}
\end{theorem}

	\begin{remark}
	We mention here  that the following, slightly weaker version of Theorem \ref{thm:mainthm} for cyclic semigroups appears in the unpublished manuscript \cite{Fa12b} of the first named author.
		Suppose that $T$ is a power-bounded operator on a Banach space $E$ not containing a copy of $\co$, $P$ is the Jacobs--de Leeuw--Glicksberg projection of the corresponding cyclic semigroup, and 
			\begin{iiv}
				\item $\{T^{n+1} - T^{n}x\mid n \in \N\}$ is relatively compact, and
				\item $((\Id-P)(\Id - T^n)x)_{n \in \N}$ converges.
			\end{iiv}
			Then $\{T^{n}x\mid n \in \N\}$ is relatively compact.
			
		An anonymous hint brought it to our knowledge that this can be deduced from the result of Ruess and Summers, mentioned earlier, as follows. Define the asymptotically almost periodic function
			\begin{align*}
				f \colon [0,\infty) \rightarrow E, \quad t \mapsto \sum_{k=0}^\infty \varphi_k(t) (T^{k+1} - T^k)x
			\end{align*}
		where
			\begin{align*}
				\varphi_k(t) = \begin{cases} 1 - 4|t-k| & \textrm{ if } |t-k| < \frac{1}{4},\\
									0 & \textrm{ else},
				\end{cases}
			\end{align*}
		for $t \geq 0 $ and $k \in \N_0$, see \cite[Remark 3.4]{RRS1991}.
		A moment's thought reveals that condition (ii) implies the Riemann integrability of the corresponding function $f_{\mathrm{s}}$. An application of the result of Ruess and Summers then yields relative compactness of $\{T^nx \mid n \in \N\}$.
		
		Though this approach reveals the connection to the Ruess--Summer version of the Bohl--Bohr--Kadets integration theorem, it cannot be used to cover the case of more general semigroups considered in this article.
	\end{remark}

	Our proof of  Theorem \ref{thm:mainthm} is inspired by Bazit's paper. The idea is to prove the theorem indirectly using the following characterization of Banach spaces (not) containing $\co$ (see \cite{BessPel} and also \cite[Thms.~6 and 8]{Di84}).
	\begin{theorem}[Bessaga--Pe{\l}czy\'nski]\label{thm:bespelcz} Let $E$ be a Banach space and for $n\in \N$ let $x_n\in E$ be vectors such that the partial sums are unconditionally bounded (i.e., $\sum_{j=1}^N x_{n_j}$ are uniformly  bounded for all subsequences $(n_j)$ of $\N$) and such that the series $\sum x_i$ is nonconvergent. Then $E$ contains a copy of $\co$.
\end{theorem}
	In addition, we need the following two simple observations.
\begin{lemma}\label{lem:firstlem}
	Let $\sgS \subseteq \LLL(E)$ be a bounded, abelian operator semigroup on a Banach space $E$. Let $x \in E$ be  such that $\sgS x$ is not relatively compact, but 
	\begin{itemize}
		\item $\sgS (\Id - T)x$ is relatively compact for every $T \in\sgS$,
		\item $(\Id - P)(\Id - \sgS)x$ is relatively compact.
	\end{itemize}		
	Then there is an infinite subset $\sgT  \subseteq \sgS$ such that
		\begin{align*}
			\inf \{\|P(Sx-Tx)\|\mid S,T \in \sgT  \textrm{ with } S \neq T\} > 0.
		\end{align*}
\end{lemma}
\noindent Note here that $Sx-Tx\in E_{\mathrm{aap}}$ holds for each $S,T\in \sgS$.
\begin{proof}
	Since $\sgS x$ is not totally bounded, there is $\delta > 0$ and a sequence $(S_n)_{n \in \N}$ in $\sgS$ such that
		\begin{align*}
			\|S_nx - S_m x\| > 2\delta \textrm{ for } n \neq m.
		\end{align*}
	Passing to a subsequence we may assume that $((\Id - P)(\Id - S_n)x)_{n \in \N}$ converges. We may even suppose that
		\begin{align*}
			\|(\Id-P)(S_nx-S_m x) \| < \delta
		\end{align*}
	for all $n,m \in \N$. This yields the claim.
\end{proof}

\begin{lemma}\label{lem:secondlem}
	Let $\sgT  \subseteq \LLL(E)$ be an infinite set of operators on a Banach space $E$. Moreover, let $F \subseteq E$ be a finite subset such that $\sgT x$ is relatively compact for every $x \in F$. For every $\varepsilon > 0$ there is an infinite subset $\EuScript{R}\subseteq \sgT $ with
		\begin{align*}
			\sup_{S,T \in \EuScript{R}} \|Sx - Tx\| \leq \varepsilon \textrm{ for every } x \in F.
		\end{align*}
\end{lemma}
\begin{proof}
 	Take a sequence of $(T_n)_{n \in \N}$ of pairwise distinct operators in $\sgT $. Passing to a subsequence we may assume that $(T_nx)_{n \in \N}$ converges for every $x \in F$ and, consequently, for $\varepsilon > 0$ there is $N \in \N$ with
 		\begin{align*}
 			\|T_nx-T_mx\| \leq \varepsilon
 		\end{align*}
 	for all $n,m \in \N$ with $n,m\geq N$.
\end{proof}

\begin{proof}[Proof of Theorem \ref{thm:mainthm}]
	We write 
		\begin{align*}
			M \defeq \sup_{S \in \overline{\sgS}} \|S\|+1 = \sup_{S \in \sgS} \|S\|+1>0
		\end{align*}
	and assume that $\sgS x$ is not relatively compact. Choose $\sgT \subseteq \sgS$ as in Lemma \ref{lem:firstlem} and set
		\begin{align*}
			\delta \defeq \inf \{\|P(Sx-Tx)\|\mid S,T \in \sgT  \textrm{ with } S \neq T\} > 0.
		\end{align*}
	We inductively construct sequences $(S_n)_{n \in \N}$ and $(T_n)_{n \in \N}$ in $\sgT$. Take $m \in \N_0$ and suppose that $S_1,\dotsc ,S_m,T_1,\dotsc ,T_m$ have already been constructed. We apply Lemma \ref{lem:secondlem} to find $S_{m+1}, T_{m+1} \in \sgT $ such that $S_{m+1} \neq T_{m+1}$ and
		\begin{align*}
			\left\|(S_{m+1} - T_{m+1}) \left(\prod_{k=1}^l S_{i_k} - \prod_{k=1}^lT_{i_k}\right)  x\right\| \leq \frac{1}{2^{m+1}M}
		\end{align*}
	for all $i_1,\dotsc,i_l \in \{0,\dotsc,m\}$.
	
	We now set $x_m \defeq T_{m}^{-1}P(T_mx - S_mx)$ for $m \in \N$. Note that $x_m$ is well-defined since $\overline{\sgS}$ restricted to $E _{\mathrm{rev}} = \rg(P)$ is a compact group. Since
		\begin{align*}
			M \|x_m\| \geq \|P(T_mx - S_mx)\| \geq \delta
		\end{align*}
	for all $m \in \N$, we obtain that the series $\sum_{m=1}^\infty x_m$ cannot converge. Now take $i_1,\dotsc , i_m \in \N$ with $i_1 < i_2 < \dotsb < i_m$. Then, by a telescopic summation argument,
		\begin{align*}
			\sum_{j=1}^m x_{i_j} = &\sum_{j = 1}^m \prod_{k=1}^{j} T_{i_k}^{-1} P\left(\left((S_{i_j} - T_{i_j})\left(\prod_{k=1}^{j-1} S_{i_k} x - \prod_{k=1}^{j-1}T_{i_k}x\right) \right)\right)\\
			&- \prod_{k=1}^m T_{i_k}^{-1} P \left(\prod_{k=1}^m S_{i_k}x - \prod_{k=1}^m T_{i_k}x\right)
		\end{align*}
		and, consequently,
			\begin{align*}
				\left\|\sum_{j=1}^m x_{i_j}\right\| \leq \sum_{j=1}^m \frac{1}{2^{i_j}} + 2M \|x\|\leq 1+ 2M\|x\|.
			\end{align*}
		This contradicts Theorem \ref{thm:bespelcz}.
\end{proof}
Under stronger conditions on the semigroup the statement of Theorem \ref{thm:mainthm} can be simplified as follows.
\begin{corollary}\label{cor:cor1}
	Let $\sgS \subseteq \LLL(E)$ be a bounded, abelian semigroup on a Banach space $E$ not containing a copy of $\co$. Assume that $\sgS$ is \emph{bounded below}, i.e., $\inf_{S \in \sgS} \|Sx\| > 0$ for every $x \in E\setminus \{0\}$. For $x \in E$ the following assertions are equivalent:
		\begin{abc}
			\item The set $\sgS x$ is relatively compact.
			\item The set $\sgS(\Id - T)x$ is relatively compact for every $T \in \sgS$.
		\end{abc}
\end{corollary}
\begin{proof}
	The assumption implies that $E_{\mathrm{aws}} = \ker(P) =\{0\}$ and therefore condition (ii) of Theorem \ref{thm:mainthm} is automatically satisfied.
\end{proof}
We also obtain the obtain the following.
\begin{corollary}\label{cor2}
	Let $\sgS \subseteq \LLL(E)$ be an abelian semigroup of isometries on a Banach space $E$ not containing a copy of $\co$. For $x \in E$ the following assertions are equivalent:
		\begin{abc}
			\item The set $\sgS x$ is relatively compact.
			\item The set $\sgS(\Id - T)x$ is relatively compact for every $T \in \sgS$.
		\end{abc}
\end{corollary}

Combining Lemma \ref{lem:A1Am} with Theorem \ref{thm:mainthm} we conclude the following result.
\begin{proposition}
	Let $E$ be a Banach space not containing $\co$ and let $\sgS$ be the operator semigroup generated by the commuting, power-bounded operators $T_1,\dotsc,T_m\in\LLL(E)$. For $x \in E$ the following assertions are equivalent:
	\begin{abc}
		\item $\sgS x$ is relatively compact.
		\item The following conditions are satisfied:
		\begin{iiv}
			\item $\sgS(\Id -T_j)x$ is relatively compact for every $j \in \{1,\dotsc,m\}$.
			\item $(\Id-P)(\Id - \sgS)x$ is relatively compact.
		\end{iiv}
	\end{abc}
\end{proposition}

\section{Geometric properties of Banach spaces}

We now give characterizations of separable Banach spaces which do not contain a copy of $\co$.  Here we call a linear operator $T \in \LLL(E)$ on a Banach space $E$ \emph{power-bounded below}, if $\inf_{n \in \N} \|T^nx\| > 0$ for every $x \in E\setminus \{0\}$. Clearly, if $T \in \LLL(E)$ is an isometry or doubly power-bounded, then $T$ is power-bounded below.
\begin{theorem}\label{thm:char1}
	For a Banach space $E$ consider the following assertions:
		\begin{abc}
			\item $E$ does not contain a copy of $\co$.	
			\item Every power-bounded operator $T \in \LLL(E)$ which is power-bounded below and has relatively compact orbits on ${\rg(\Id-T)}$ has relatively compact orbits on $E$.
			\item Every power-bounded operator $T \in \LLL(E)$ which is power-bounded below  and has relatively compact orbits on ${\rg(\Id-T)}$ is mean ergodic on $E$.
			\item Every doubly power-bounded operator $T \in \LLL(E)$ with relatively compact orbits on ${\rg(\Id-T)}$ has relatively compact orbits on $E$.
			\item Every doubly power-bounded operator $T \in \LLL(E)$ with relatively compact orbits on ${\rg(\Id-T)}$ is mean ergodic on $E$.
		\end{abc}
	Then {\upshape (a)} $\Rightarrow$ {\upshape(b)} $\Leftrightarrow$ {\upshape(c)} $\Rightarrow$ {\upshape(d)} $\Leftrightarrow$ {\upshape(e)}. If $E$ is separable, then all the  assertions are equivalent.
\end{theorem}
\begin{proof}
	Using Lemma \ref{lem:A1Am} and Corollary \ref{cor:cor1} we obtain that (a) implies (b). The other implications between assertions (b)--(e) are either trivial or follow directly from Lemma \ref{lem:A1Am} and Theorem \ref{thm:charmeanergodic1}. Now assume that $E$ is separable.
	By Sobczyk's theorem, see \cite{sobczyk1941}, or e.g., \cite[Sec.~2.5]{AlbiacKalton}, if $\co$ is a closed subspace in a separable Banach space, then it is complemented in there. Using Example \ref{exa:1} and the identity operator on a complement of $\co$ shows that (e) implies (a). 
\end{proof}
A different but related characterization is the following.
\begin{theorem}\label{thm:char2}
	For a Banach space $E$ consider the following assertions:
		\begin{abc}
			\item $E$ does not contain a copy of $\co$.
			\item For every power-bounded operator $T \in \LLL(E)$ which is power-bounded below the following assertions are equivalent: 
				\begin{itemize}
					\item $T$ has relatively compact orbits on ${\rg(\Id-T^m)}$  for some $m \in \N$.
					\item $T$ has relatively compact orbits on ${\rg(\Id-T^m)}$  for each $m \in \N$.
				\end{itemize}
			\item For every doubly power-bounded operator $T \in \LLL(E)$ the following assertions are equivalent: 
				\begin{itemize}
					\item $T$ has relatively compact orbits on ${\rg(\Id-T^m)}$  for some $m \in \N$.
					\item $T$ has relatively compact orbits on ${\rg(\Id-T^m)}$  for each $m \in \N$.
				\end{itemize}
		\end{abc}
	Then {\upshape (a)} $\Rightarrow$ {\upshape (b)} $\Rightarrow$ {\upshape (c)}. If $E$ is separable, then all three statements are equivalent.
\end{theorem}
The following example, which is exhibited on $\ce$ but can be transformed easily onto $\co$, is needed for the proof.
\begin{example}\label{exa:A1Am}
Let $E\defeq\ce$ and for $m\in\N$, $m\geq 2$ fixed let $(a_n)_{n \in \N}$ be a sequence in $\C$ converging to an $m$th root of unity $\xi$ such that $|a_n| = 1$ and $a_n \neq \xi$ for every $n \in \N$. We consider the induced multiplication operator $T \in \LLL(\ce)$ (cf.{} Example \ref{exa:1}) which is an invertible isometry. We then have $E_{\mathrm{aap}}=E_{\mathrm{rev}}\subseteq \co$. Since $\xi^{-1}T$ (and thus also $T$) does not have relatively compact orbits (see Example \ref{exa:1}) we even obtain $E_{\mathrm{aap}}=E_{\mathrm{rev}} = \co$. Since  $\rg(\Id-T^m)\subseteq \co$ and $\one - (a_n)_{n \in \N} \in \rg(\Id-T)\setminus \co$ (with $\one$ the constant $1$ sequence), we obtain that $T$ has relatively compact orbits on ${\rg(\Id-T^m)}$ but not on ${\rg(\Id-T)}$.  
\end{example}
\begin{proof}[Proof of Theorem \ref{thm:char2}]
	Assume that $E$ does not contain a copy of $\co$. Moreover, let $T \in \LLL(E)$ be a power-bounded operator which is power-bounded below and $m \in \N$ such that $T$ has relatively compact orbits on ${\rg(\Id-T^m)}$. Then $T^m$ has also relatively compact orbits on ${\rg(\Id-T^m)}$ and therefore $T^m$ has relatively compact orbits on $E$ by Theorem \ref{thm:char1}. But then $T$ also has relatively compact orbits on $E$. In fact, if $x \in E$, we write 
	\begin{align*}
		B \defeq \{T^{km}x\mid k \in \N_0\}
	\end{align*}
	 and obtain
		\begin{align*}
			\{T^nx\mid n\in \N\} \subseteq B \cup T(B) \cup \dotsb \cup T^{m-1}(B)
		\end{align*}
	Since the set on the right-hand side is relatively compact, this proves the implication \enquote{(a) $\Rightarrow$ (b)}. The implication  \enquote{(b) $\Rightarrow$ (c)} is trivial.
	
	Finally, if $E$ is separable, we once again use that $\co$ is complemented in $E$ (see  \cite{sobczyk1941}). Using Example \ref{exa:A1Am} on $\co$ and the identity operator on a complement of $\co$ in $E$, we obtain the implication \enquote{(c) $\Rightarrow$ (a)}.
\end{proof}

Finally, we come to our main result.
\begin{theorem}\label{thm:char3}
	For a Banach space $E$ with a Schauder basis the following assertions are equivalent:
		\begin{abc}
			\item $E$ is reflexive.
			\item Every power-bounded operator $T \in \LLL(E)$ which has relatively compact orbits on $\rg(\Id-T)$ is mean ergodic on $E$.
			\item Every power-bounded operator $T \in \LLL(E)$ which has relatively compact orbits on $\rg(\Id-T)$ has relatively compact orbits on $E$.
			\item Every power-bounded operator  $T \in \LLL(E)$  which has relatively weakly compact orbits on $\rg(\Id-T)$ has relatively weakly compact orbits on $E$.
			\item Every power-bounded operator  $T \in \LLL(E)$  which has relatively weakly  compact orbits on $\rg(\Id-T)$ is mean ergodic on $E$.
			\item Every power-bounded operator $T \in \LLL(E)$ with countable peripheral spectrum $\sigma(T) \cap \T$ has relatively compact orbits on $E$.
			\item Every power-bounded operator $T \in \LLL(E)$ with peripheral spectrum $\sigma(T) \cap \T \subseteq \{1\}$ is mean ergodic on $E$.	
		\end{abc}
\end{theorem}
\begin{proof}
	Assume that $E$ is reflexive. We then obtain (c) by Theorem \ref{thm:char1}. Moreover, (d) holds since all power-bounded operators on a reflexive space have relatively weakly compact orbits by the Banach--Alaoglu theorem, and (f) follows from Remark \ref{spectralcrit} (iii). Therefore (a) implies (c), (d) and (f).\smallskip
	
	 The implications \enquote{(c) $\Rightarrow$ (b)}, \enquote{(e) $\Rightarrow$ (b)} and \enquote{(f) $\Rightarrow$ (g)} are trivial. That (d) implies (e) follows by the Yosida--Kakutani mean ergodic theorem, see \cite{Yo38, Ka38} or \cite[Thm. 8.22]{EFHN}, and the implication \enquote{(b) $\Rightarrow$ (g)} follows from the Katznelson--Tzafriri theorem, see \cite{KaTz86}.\smallskip
	 
	 To finish the proof, it thus suffices to show that the implication \enquote{(g) $\Rightarrow$ (a)} holds.  Assume that $E$ is not reflexive. By \cite[Theorem 1]{FLW01} we find a power-bounded operator $T \in \LLL(E)$ which is not mean ergodic. Passing to an equivalent norm on $E$ we may assume that $T$ is a contraction. We now consider $S \defeq \frac{1}{2}(I+T)$. For every $\lambda \in \C$ with $|\lambda| =1 \neq \lambda$ we have $|\lambda - \frac{1}{2}| > \frac{1}{2}$. This implies that $\lambda - S = ((\lambda - \frac{1}{2}) - \frac{1}{2}T)$ is invertible. Thus, $\sigma(S) \cap \T \subseteq \{1\}$ and consequently (g) cannot hold. 
\end{proof}

\section*{Acknowledgement}
\noindent We are indebted to Jochen Gl\"uck and Michael Lin for their insightful comments and suggestions. In particular, we are grateful to Michael Lin for suggesting statements (d), (e) and (g) as well as a new proof for Theorem \ref{thm:char2} allowing to remove the assumption that the space admits an \emph{unconditional} Schauder basis.

\bibliographystyle{amsabbrv}
\bibliography{refs-lin,bibliography}
\parindent0pt
\end{document}